\documentclass[12pt]{amsart}
\usepackage{amsmath}
\usepackage{enumerate}
\usepackage{amssymb}
\usepackage{amsbsy}
\usepackage{amsfonts}
\usepackage[all]{xy}
\input cyracc.def

\theoremstyle{plain}
\newtheorem{thm}{Theorem}[section]

\newtheorem{lem}[thm]{Lemma}
\newtheorem{cor}[thm]{Corollary}
\theoremstyle{definition}
\newtheorem{dfn}[thm]{Definition}
\newtheorem{ex}[thm]{Example}

\newtheorem{rmk}[thm]{Remark}
\numberwithin{equation}{section}

\begin{document}

\title{The converse of a theorem by Bayer and Stillman}

\author{HyunBin Loh}
\address[Loh]{Department of Mathematics, POSTECH, Pohang, Gyungbuk 790-784, R. O. Korea}
\email{hyunbin@postech.ac.kr}

\maketitle

\begin{abstract}
Bayer-Stillman showed that $reg(I) = reg(gin_\tau(I))$ when $\tau$ is the graded reverse lexicographic order. We show that the reverse lexicographic order is the unique monomial order $\tau$ satisfying $reg(I) = reg(gin_\tau(I))$ for all ideals $I$. We also show that if $gin_{\tau_1}(I) = gin_{\tau_2}(I)$ for all $I$, then $\tau_1 = \tau_2$.
\end{abstract}

\section{Introduction}

If we have an ideal $I$ and a monomial term order $\tau$, then there is a Zariski open dense subset $U$ of coordinate transformations where the initial ideal is stable \cite{Ga}. This initial ideal is called the generic initial ideal denoted $gin_\tau(I)$ or simply $gin(I)$ if the monomial order is specified before. It can be shown that the generic initial ideal is a Borel-fixed monomial ideal. Then by the good combinatorial properties of Borel-fixed monomial ideals, we can analyze the structure of $gin(I)$. For example, the minimal free resolution is given by the Eliahou-Kervaire theorem and the regularity is given by the maximum degree of a minimal generator \cite{BaSt}. Also, the Betti numbers of an ideal $I$ are bounded by the Betti numbers of generic initial ideals \cite{Bigatti} \cite{ConcaKos}.

A well known result of Conca on generic initial ideals is that if $I$ is Borel-fixed, then $gin_\tau(I)=I$ for any $\tau$ \cite{ConcaKos}. There are more results on specific monomial ideals \cite{Chardin} \cite{ConcaKos} \cite{MoSo}. In the case where $I$ is not a monomial ideal however, these methods are not directly applicable. In this paper, we generalize the notion of $\tau$-segment ideals in \cite{ConcaSidman}. We show that if $in_\tau(I)$ is a $\tau$-segment ideal, then $gin_\tau(I) =in_\tau(I)$. Here, we do not require $I$ to be a monomial ideal. Consequently, we will construct an ideal which has different generic initial ideals for two given monomial orders. This implies that the generic initial ideals fully characterize monomial term orders.

When regarding the degree complexity of an ideal, the regularity of an ideal is a good invariant. An ideal $I$ is $m$-regular if the $j^{th}$ syzygy module of $I$ is generated in degrees $\le m+j$, for all $j \ge 0$. The regularity of $I$, $reg(I)$, is defined as the least $m$ for which $I$ is $m$-regular \cite{EiGo}. Since the graded Betti numbers are upper-semicontinuous in flat families, we have $reg(in_\tau(I)) \ge reg(I)$ for any $\tau$ \cite{Pa}. In general coordinates and graded reverse lexicographic order(rlex), Bayer and Stillman showed that $reg(in_\textrm{rlex}(I)) = reg(I)$ \cite{BaSt}. In this aspect, rlex is an optimal order for the computation of Gr\"obner Bases. Bayer and Stillman also suggested a method of refining monomial orders by the reverse lexicographic order, which will give faster computation \cite{BaSt2}. We show that for any other monomial order $\tau$ besides rlex, there exists an ideal $I$ such that $reg(gin_\tau(I))> reg(I)$. This implies that the graded reverse lexicographic order is the unique optimal monomial order that gives minimum regularity.

\section*{Acknowledgement}

The author would like to thank his adviser Donghoon Hyeon for teaching the statement of the main theorem, for suggesting a general idea of the proof, and for giving valuable comments to improve the quality of the paper. The author would like to thank Hwangrae Lee for suggesting the idea of lemma 3.6., which helped to shorten the proofs considerably. The author would also like to thank Jeaman Ahn for helpful conversations. The author was supported by the following grants funded by the government of Korea: NRF grant NRF-2013H1A8A1004216.

\section{Notation and Terminology}

Let $S=K[x_1, \dots ,x_n]$ be a polynomial ring over an algebraically closed field $K$ with $char K =0$. Let $\bf{x^\alpha}=$$x_1^{\alpha_1} \dots x_n^{\alpha_n}$ be the vector notation. For a homogeneous ideal $I$, let $\mathcal{G}(I)$ be a Gr\"obner basis of $I$.

In this paper, we assume all monomial orders to be graded multiplicative orders with $x_1>x_2>\dots>x_n$. A monomial order $\tau$ is graded if $deg(f) > deg(g)$ implies $f >_\tau g$. A monomial order $\tau$ is multiplicative if $f >_\tau g$ implies $fh >_\tau gh$. Then $fh >_\tau gh$ also implies $f >_\tau g$. Let rlex denote the graded reverse lexicographic order and lex denote the graded lexicographic order. 

Let $B= \{f_1,\dots,f_k\} \subset S$ be a set and $V=K\langle f_1,\dots,f_k \rangle \in S_d$ be the vector space spanned by $B$. Then, define $in_\tau(B)=\{ in_\tau(f_1),\dots,in_\tau(f_k) \}$ and $in_\tau (V) = K \langle in_\tau(f) | f \in V \rangle$.

\begin{dfn}
Let $M$ be a finitely generated graded $S$-module and
$$0 \rightarrow \oplus_j S(-a_{l_j}) \rightarrow \cdots \rightarrow \oplus_j S(-a_{1_j}) \rightarrow \oplus_jS(-a_{0_j}) \rightarrow M \rightarrow 0$$ be a minimal graded free resolution of $M$. We say that $M$ is $d$-regular if $a_{ij} \le d+i$ for all $i,j$. Let the regularity of $M$, denoted $reg(M)$ by the least $d$ such that $M$ is $d$-regular.
\end{dfn}

\begin{rmk}
The regularity of an ideal $I$ is defined by the minimal free resolution of the following form. $$0 \rightarrow \oplus_j S(-a_{l_j}) \rightarrow \cdots \rightarrow \oplus_j S(-a_{1_j}) \rightarrow \oplus_jS(-a_{0_j}) \rightarrow I \rightarrow 0$$
Then the minimal free resolution of $M=S/I$ follows from that of $I$. $$0 \rightarrow \oplus_j S(-a_{l_j}) \rightarrow \cdots \rightarrow \oplus_j S(-a_{1_j}) \rightarrow \oplus_jS(-a_{0_j}) \rightarrow S \rightarrow S/I \rightarrow 0$$
Hence have $reg(S/I) = reg(I) + 1$. Note that if $I$ has a minimal generator of degree $d$, then $reg(I) \ge d$.
\end{rmk}

\section{Generic initial ideals and $\tau$-segment ideals}

The notion of generic initial ideals was introduced by Galligo \cite{Ga}. He showed that generic initial ideals have a good combinatorial property called Borel-fixedness. Since then, generic initial ideals have been studied extensively in commuative algebra and geometry. We introduce the theorem of Galligo. For a more detailed introduction, see \cite{Eisenbud}.

\begin{dfn}
A monomial ideal $I$ is Borel-fixed if $m \in I$ and $m \frac{x_i}{x_j} \in S$ for $i<j$ implies $m \frac{x_i}{x_j} \in I$.
\end{dfn}

\begin{thm} [Galligo]
For a given ideal $I$ and monomial term order $\tau$, there exists a Zariski open subset $U$ of $GL(V)$ such that $in_\tau(g(I))$ is the same for all $g \in U$. We define $gin_\tau(I):=in_\tau(g(I))$ for $g \in U$.
\end{thm}

Galligo also showed that generic initial ideals are Borel-fixed. Conca showed the converse: if $I$ is a Borel-fixed ideal, then $gin_\tau(I) = I$ for any $\tau$. We will say that $I$ is in general coordinates in the way that $id \in U$ where $in_\tau(g(I))$ is stable for $g \in U$. However, if $I$ is not a monomial ideal, we cannot use similar methods because there is no concept of Borel-fixedness. Taking the initial ideal also does not work well because syzygy computations are not preserved under coordinate transformations. We extend Conca's results to some non-monomial ideals by introducing the notion of $\tau$-segment ideals. This is a generalization of $Seg_\tau(I)$ introduced in \cite{ConcaSidman} that we do not require the ideal to be a $\tau$-segment in every degree. By definition, a $\tau$-segment ideal is always an ideal where $Seg_\tau(I)$ may not be an ideal. Adopting our definition, we show that if $in_\tau(I)$ is a $\tau$-segment ideal, we have $gin_\tau(I) = in_\tau(I)$.

\begin{dfn}
Let $B = \{f_1,\dots,f_k\}$ be a set of monomials with $deg(f_i) = d_i$. If $g \in B$ for all monomials $g \in S$ such that $deg(g)=d_i$ for some $i$ and $g >_\tau f$ for some $f \in B$, call $B$ a $\tau$-segment. If an ideal $I=(f_1,\dots,f_k)$ is generated by a $\tau$-segment $B=\{f_1,\dots,f_k\}$, then call $I$ a $\tau$-segment ideal.
\end{dfn}

\begin{ex}
Let $S=K[x,y,z]$ and $\bf{w}$$=(10,5,3)$ be a graded weight order with tie breaking by lex. The ideal $I=( x^2, xy, y^5) \subset S$ is a $\bf{w}$-segment ideal generated in degrees $2$ and $5$. The bases of $I_2 = K\langle x^2,xy \rangle$ and $I_5 = K\langle f \ | \ deg(f)=5,  f \ge $$\bf{_w}$ $ xyz^3 \rangle$ are both $\bf{w}$-segments. Note that $I_3, I_4$ are not $\bf{w}$-segments since $y^3 >$$\bf{_w}$ $xyz \in I_3$ and $y^4 >$$\bf{_w}$$,xyz^2 \in I_4$ but $y^3,y^4 \not\in I$.
\end{ex}

When $\tau$ is the graded lexicographic order, the lex-segment ideals have good combinatorial properties \cite{Peeva}. If $I$ is a lex-segment ideal, then the generating set of $I_d$ is a lex-segment for every $d$. There follows a one-to-one correspondence with lex-segment ideals and Hilbert functions satisfying a particular growth criterion by Gotzmann. For $\tau \neq$ lex, there always exists some $d$ where $I_d$ is not a $\tau$-segment. For general $\tau$, the $\tau$-segments and $\tau$-segment ideals have the following property.

\begin{lem}
Let $\tau$ be any graded monomial order. \\ (a) A $\tau$-segment is Borel fixed. \\ (b) A $\tau$-segment ideal is Borel fixed.
\end{lem}

\begin{proof}
(a) Let $B$ be a $\tau$-segment. Let $f \in B$ and $f \frac{x_i}{x_j} \in S$ for $i<j$. Then we have $f \frac{x_i}{x_j} >_\tau f$ since $x_j f \frac{x_i}{x_j} = x_i f >_\tau x_j f$. By the definition of $\tau$-segments, $f \frac{x_i}{x_j} \in B$. So $B$ is Borel-fixed. \\
(b) Let $I=(f_1, \, f_k)$ be a $\tau$-segment ideal. Suppose $F=h f_t$ is a monomial in $I$ for some $t$ and $F \frac{x_i}{x_j} = h f_t \frac{x_i}{x_j} \in S$ for $i<j$. If $f_t \frac{x_i}{x_j} \in S$, we have $f_t \frac{x_i}{x_j} \in I$ by the definition of $\tau$-segment ideals. Otherwise if $f_t \frac{x_i}{x_j} \not\in S$, we have $h \frac{x_i}{x_j} \in S$. Therefore, $F= h \frac{x_i}{x_j} f_t \in I$.
\end{proof}

Let $in_\tau(I)$ be a $\tau$-segment ideal for a homogeneous ideal $I$. Since $\tau$-segment ideals are Borel-fixed, $in_\tau(I)$ is already in general coordinates. Moreover, if $in_\tau(I)$ is a $\tau$-segment, then $gin_\tau(I) = in_\tau(I)$. This means that $I$ is also in general coordinates.

\begin{lem} \label{thelem}
If $in_\tau(I)$ is a $\tau$-segment ideal, then $gin_\tau(I) = in_\tau(I)$. 
\end{lem}

\begin{proof}
We shall prove that $gin(I)_d = in(I_d)$ for all $d$. Let $in_\tau(I)$ be a $\tau$-segment ideal with minimal generators in degree $d_1, \dots, d_t$. 

First suppose that $d=d_i$ for some $i$. Let $M_1 > M_2 > \dots$ be the total ordering of degree $d$ monomials with respect to $\tau$. Since $in(I)$ is a $\tau$-segment ideal, we have $in(I)_d = \langle M_1, \dots , M_r \rangle$ for some $r$. Then, $\wedge^r (in(I_d)) = \langle M_1 \wedge \dots \wedge M_r \rangle$. Let $g=[g_{ij}] \in GL(S_1)$ be a coordinate transformation. Then we have $\wedge ^r (g(I))_d = \langle g(M_1) \wedge \dots \wedge g(M_r) \rangle = \langle P_d(g_{11},\dots,g_{nn})$ $M_1 \wedge \dots \wedge M_r$ + lower terms$\rangle$. However, $\wedge^r(in(I_d)) = \langle M_1 \wedge \dots \wedge M_r \rangle$, which is the largest standard exterior monomial in $\wedge^r(S_d)$. This means that the coefficient polynomial $P_d(g_{11},\dots, g_{nn})$ of $M_1 \wedge \dots \wedge M_r$ is nonvanishing for $g=id$. Hence $U_d = \{g |P_d(g_{11},\dots,g_{nn}) \neq 0\}$ is a nonempty Zariski open subset where $in(g(I))$ is stable. Therefore $gin(I)_d = in(I_d)$.

Now let $d \neq d_1, \dots, d_t$. Since there are no Gr\"obner bases of degree $d$, we have $in(I_d) = in(I_{d-1}) S_1$. Then, $gin(I)_d \supset gin(I)_{d-1} S_1 = in(I_{d-1}) S_1 = in(I_d)$. Since $in(I)$ and $gin(I)$ have the same dimension in every degree, we have $gin(I)_d = in(I_d)$. Since $gin(I)_d = in(I_d)$ for every $d$, we conclude that $gin(I) = in(I)$.
\end{proof}

\begin{rmk}
If $in(I)$ is Borel-fixed, $gin(I)$ may differ from $in(I)$. Let $S=K[x,y,z]$ and $I=(x^3,x^2y+xy^2,x^2z)$. Then $in_\mathrm{rlex}(I)=(x^3,x^2y,x^2z,xy^3,xy^2z)$ but $gin_\mathrm{rlex}(I)=(x^3,x^2y,xy^2,x^2z^2)$.
\end{rmk}

Now we have a class of ideals which are already in general coordinates. We use this lemma for the constructions of ideals showing our main results. The following theorem shows that generic initial ideals fully characterize monomial orders.

\begin{thm}
$gin_{\tau_1}(I)=gin_{\tau_2}(I)$ for all ideals $I \subset S$, if and only if $\tau_1 = \tau_2$.
\end{thm}

\begin{proof}
One way is trivial. For the other way, we show that if $\tau_1 \neq \tau_2$ then there exists some $I$ such that $gin_{\tau_1}(I) \neq gin_{\tau_2}(I)$. 
Let $x_1^d=M_1>_{\tau_1} M_2>_{\tau_1}\dots$ be the total ordering of degree $d$ monomials with respect to $\tau_1$ and $x_1^d=M_1'>_{\tau_2} M_2'>_{\tau_2}\dots$ be the total ordering of degree $d$ monomials with respect to $\tau_2$. Let $k$ be the least integer such that $M_k \neq M_k'$. Define the ideal $I=(M_1,\dots,M_{k-1},M_k+M_k')$. 

By symmetry, it suffices to show that $gin_{\tau_1}(I) = (M_1, \dots, M_{k-1}, M_k)$. We use Buchberger's algorithm on $I$. Since $I$ is generated by degree $d$ homogeneous elements, all syzygies have degree larger than $d$. Then, $in_{\tau_1}(I)_d$ is generated by the initial parts of the degree $d$ Gr\"obner bases. These are just the initial terms of the generators of $I$. Then $in_{\tau_1}(I)_d = \langle M_1,\dots,M_k \rangle$. Since $M_1,\dots,M_k$ are the largest $k$ monomials in degree $d$ with respect to $\tau_1$, $in_{\tau_1}(I_d)$ is a $\tau_1$-segment. By Lemma \ref{thelem}, we have $gin_{\tau_1}(I)_d = in_{\tau_1}(I_d) = \langle M_1,\dots,M_k \rangle$.
\end{proof}

\section{The reverse lexicographic order}

We have $reg(I) = reg(g(I))$ for any ideal $I$ and a coordinate transformation $g \in GL(S_1)$ because the Betti tables of $I$ and $g(I)$ coincide. However, taking the initial ideal does not commute with coordinate transformation because syzygy calculations are not preserved under coordinate transformations. 

Where $reg(I) \le reg(in_\tau(I))$ for any order $\tau$, the following theorem of Bayer and Stillman shows that the graded reverse lexicographic order gives the lowest possible regularity for generic initial ideals.

\begin{thm} [Bayer-Stillman] \cite{BaSt}
If $I$ is a homogeneous ideal, then $reg(I)=reg(gin_\mathrm{rlex}(I))$.
\end{thm}

Thus the graded reverse lexicographic order is an optimal order in Gr\"obner base calculation. Conversely, we show that if $reg(I) = reg(gin_\tau (I))$ for all ideals $I \subset S$, then $\tau=\mathrm{revlex}$. This implies the unique optimality of the graded reverse lexicographic order. However, this does not show that general coordinates give the lowest regularity. If $I = (x^2+y^2,xyz) \subset S=K[x,y,z]$, we have $reg(in_\mathrm{lex}(I))= 4$ but $reg(gin_\mathrm{lex}(I))= 5$. Before we introduce the main theorem, we first characterize the graded reverse lexicographic order.

\begin{lem} \label{revlex}
$\tau = \mathrm{rlex}$ if and only if $x_{k-1}^{d+1} > x_1^d x_k$ for all $k, d$.
\end{lem}

\begin{proof}
One way is trivial. We show that if $x_{k-1}^{d+1} >_\tau x_1^d x_k$ for all $k$, then $\tau$ is the reverse lexicographic order. Let $f=x_1^{a_1} x_2^{a_2} \dots x_n^{a_n}$, $g=x_1^{b_1} x_2^{b_2} \dots x_n^{b_n}$ be degree $d+1$ polynomials. If $K$ is the largest $i$ such that $a_i \neq b_i$, Let $a_K < b_K$. We show that $f>_\tau g$.
 	
Since $\tau$ is multiplicative, the term order is preserved under factoring out common terms. We factor out $c=x_K^{a_K}$. Any monomial order $\tau$ with $x_1>_\tau \dots>_\tau x_n$ includes the Borel order in the way that if $M>_{Borel} N$ then $M>_\tau N$. We have 
$f/c = x_1^{a_1} \dots x_{K-1}^{a_{K-1}} >_\tau  x_{K-1}^{d+1-a_K}  >_\tau x_1^{d-a_K} x_K >_\tau x_1^{b_1} x_2^{b_2} \dots x_K^{b_K - a_K} = g/c$. Therefore, $f>_\tau g$. This is the defining property of the reverse lexicographic order. Hence $\tau$ is the reverse lexicographic order.
\end{proof}

\begin{lem} [Conca] \cite{ConcaKos} \label{lemconca}
Let $I$ be a Borel-fixed ideal and let $m_1, \dots, m_k$ be its monomial generators. Let $g \in GL(K)$ be a generic matrix. Then $g(I)$ is generated by polynomials $f_1, \dots, f_k$ of the form $f_i = m_i + h_i$ such that the monomials in $h_i$ are smaller than $m_i$ in the Borel-order. The polynomials $f_1,\dots,f_k$ form a Gr\"obner basis of $g(I)$ with respect to any term order.
\end{lem}

Now we prove our main theorem.

\begin{thm}
If $reg(gin_{\tau}(I))=reg(gin_{rlex}(I))$ for all ideals $I \subset S$, then $\tau= \textrm{rlex}$.
\end{thm}

\begin{proof}
Suppose $\tau \neq \textrm{rlex}$. By Lemma \ref{revlex}, there exists some $k,d$ such that $x_1^d x_k > x_{k-1}^{d+1}$. We show that $reg(gin_{\textrm{rlex}}(I)) \neq reg(gin_\tau(I))$ for the ideal $I = (x_1^{d+1},\dots,x_{k-2} x_{k-1}^d, x_{k-1}^{d+1} + x_1^d x_k)$. This ideal $I$ is generated by $x_{k-1}^{d+1} + x_1^d x_k$ and all degree $d+1$ monomials in $K[x_1,\dots,x_{k-1}]$ except $x_{k-1}^{d+1}$.

First, consider the graded reverse lexicographic case. Let $x_1^{d+1} = M_1 >_\textrm{rlex} M_2 >_\textrm{rlex} \dots >_\textrm{rlex} M_{L+1} = x_{k-1}^{d+1}$ be the total ordering of degree $d+1$ monomials in $K[x_1,\dots,x_{k-1}]_{d+1}$. Then we can write $I = (M_1,\dots,M_L, x_{k-1}^{d+1} + x_1^d x_k)$. We use Buchberger's algorithm and show that no syzygy is added to the Gr\"obner base. The syzygies for the first $L$ generators are $0$. Also for any possible syzygy $S=f_1 M_i-f_2(x_{k-1}^{d+1}+x_1^d x_k)=f_2 x_1^d x_k$, we have $f_2 x_1^d x_k \in (x_1,\dots,x_{k-1})^{d+1}$ since $f_2 | M_i$ and $M_i \in (x_1,\dots,x_{k-1})$. Therefore, $\{M_1,\dots, M_L, x_{k-1}^{d+1} + x_1^d x_k \}$ is a Gr\"obner base of $I$. Consequently, $in_\mathrm{rlex}(I)=(x_1,\dots,x_{k-1})^{d+1}$. Since this is a rlex-segment ideal, we have $gin_\textrm{rlex}(I) = (x_1,\dots,x_{k-1})^{d+1}$ by Lemma \ref{thelem}. Then $reg(gin_\mathrm{rlex}(I))=d+1$, which is the maximum degree of a minimal generator of $gin_\textrm{rlex}(I)$.

Now, let $\tau \neq \textrm{rlex}$ with $x_1^d x_k >_\tau x_{k-1}^{d+1}$. Let $I' = (M_1, \dots, M_L)$ and $M_0=x_1^d x_k + x_{k-1}^{d+1}$. Then, $in_\tau(g(\wedge^{r+1} I_{d+1}) ) = in_\tau(g(M_1)\wedge g(M_2)\wedge \dots \wedge g(M_L) \wedge g(M_0))$. Take $g$ a general coordinate for $I_{d+1}$ and $I'_{d+1}$. Since $I'$ is Borel-fixed, $in_\tau(g( \wedge^r I'_{d+1})) = M_1 \wedge \dots \wedge M_L$. This means that $g(M_1) \wedge \dots \wedge g(M_L) = P(g) (M_1 \wedge \dots \wedge M_L) + \mathrm{(lower \ terms)}$ for $P(g) \neq 0$. We take $g$ generic such that $g(M_0)$ has nonzero coefficients for all degree $d+1$ monomials. This can be done by expanding $g(M_0)$ and taking the coordinate transformation avoiding the zero locus of each coefficient of the monomial terms. Since $x_1^d x_k$ is the largest degree $d+1$ monomial besides $M_1,\dots,M_L$, we obtain $in_\tau(g(\wedge^{r+1} I_{d+1}) )=M_1 \wedge \dots \wedge M_L \wedge x_1^d x_k$. This exterior monomial may not be in standard form because we don't know the order in $\tau$.

We observe that $S=x_{k-1}^{d+2} = x_{k-1} (x_1^d x_k + x_{k-1}^{d+1}) - x_k (x_1^d x_{k-1}) \in I$. Then we add this redundant basis so that $I = (M_1,\dots,M_L,M_0,x_{k-1}^{d+2})$. Let $J = (M_1, \dots, M_L, x_{k-1}^{d+2})$ then $J$ is Borel-fixed. By Lemma \ref{lemconca}, $\mathcal{G}(g(J)) = \{ M_1+N_1, \dots, M_L+N_L, x_{k-1}^{d+2} + N_{L+1} \}$ where the $N_i$ are linear sums of terms smaller than $M_i$ in Borel-order. Then we have $g(I) = (M_1+N_1, \dots, M_L+N_L, g(M_0), x_{k-1}^{d+2}+N_{L+1})$.

Since we have shown that $in_\tau(g(\wedge^{r+1} I_{d+1}) )=M_1 \wedge \dots \wedge M_L \wedge x_1^d x_k$, we rewrite this as $g(I) = (M_1+N_1, \dots, M_L+N_L, x_1^d x_k + N_0 , x_{k-1}^{d+2}+N_{L+1})$. The syzygy $S = x_{k-1}(x_1^d x_k + x_{k-1}^{d+1}) - x_k (x_1^d x_{k-1}) = x_{k-1}^{d+2}$ in $I$ is not reducible by $M_1, \dots, M_L, M_0$ using $\tau$. Since the initial terms of the generators of $g(I)$ and $I$ coincide, we also cannot reduce $x_{k-1}^{d+2}+N_{L+1}$ by lower degree generators of $g(I)$. Hence, this is a proper Gr\"obner base of $g(I)$. Consequently, $gin_\tau(I) = in_\tau(g(I))$ has a generator of degree $d+2$ and therefore has regularity $\ge d+2$.
\end{proof}

\begin{ex}
Let $S=K[x_1,\dots,x_6]$ and $I = (x_1^{3},x_1^2 x_2, x_1 x_2^2, x_2^3 + x_1^2 x_3)$. Then, $gin_\mathrm{lex}(I)=(x_1^{3},x_1^2 x_2, x_1 x_2^2, x_1^3 x_3)+(x_2^{4})$ and $gin_\textrm{rlex}(I) = (x_1^3, x_1^2 x_2, x_1 x_2^2 ,x_2^3)$. Hence the regularities are $reg(gin_\mathrm{lex}(I)) = 4, reg(gin_\mathrm{rlex}(I))=3$.
\end{ex}

Using the theorem, we directly obtain the converse statement of Bayer and Stillman.

\begin{cor}
If $reg(gin_\tau(I))=reg(I)$ for all ideals $I \subset S$, then $\tau = \textrm{rlex}$.
\end{cor}

\begin{proof}
This follows from the result of Bayer-Stillman: $reg(gin_{rlex}(I))=reg(I)$ \cite{BaSt}.
\end{proof}

\end{document}